\documentclass[aop]{imsart}

\RequirePackage{amsthm,amsmath,amsfonts,amssymb,latexsym,amsfonts,amscd}
\RequirePackage[numbers]{natbib}
\RequirePackage{enumerate}
\RequirePackage[colorlinks,citecolor=blue,urlcolor=blue]{hyperref}
\RequirePackage{graphicx}
\RequirePackage{subfigure}
\RequirePackage{mathrsfs}
\RequirePackage{footmisc}
\usepackage{tikz}
\usetikzlibrary{shapes.geometric, arrows}
\usetikzlibrary{decorations,decorations.markings}
\usetikzlibrary{arrows,shapes,chains}

\startlocaldefs

\theoremstyle{plain}

\newtheorem{theorem}{Theorem}[section]
\newtheorem{lemma}[theorem]{Lemma}

\theoremstyle{remark}

\newtheorem*{fact}{Fact}
\makeatletter
\newcommand*{\rom}[1]{\expandafter\@slowromancap\romannumeral #1@}
\makeatother
\newcommand{\norm}[1]{\left\lVert #1 \right\rVert}

\newcommand{\eqnsection}{
\renewcommand{\theequation}{\thesection.\arabic{equation}}
    \makeatletter
    \csname  @addtoreset\endcsname{equation}{section}
    \makeatother}
\eqnsection

\newtheorem*{remark*}{Remark}

\allowdisplaybreaks[4]
\DeclareFontFamily{U}{mathx}{\hyphenchar\font45}
\DeclareFontShape{U}{mathx}{m}{n}{
      <5> <6> <7> <8> <9> <10>
      <10.95> <12> <14.4> <17.28> <20.74> <24.88>
      mathx10
      }{}
\DeclareSymbolFont{mathx}{U}{mathx}{m}{n}
\DeclareFontSubstitution{U}{mathx}{m}{n}
\DeclareMathAccent{\widecheck}{0}{mathx}{"71}
\DeclareMathAccent{\wideparen}{0}{mathx}{"75}

\DeclareFontFamily{U}{lasy}{}
\DeclareFontShape{U}{lasy}{m}{n}{
  <-5.5> lasy5
  <5.5-6.5> lasy6
  <6.5-7.5> lasy7
  <7.5-8.5> lasy8
  <8.5-9.5> lasy9
  <9.5-> lasy10
}{}

\def\r{{\mathbb R}}

\def\p{{\mathbb P}}

\def\N{{\mathbb N}}

\def\cL{\mathcal L}

\def\bG{\mathbb G}

\def\bn{\mathbf n}
\def\bq{\mathbf q}

\endlocaldefs

\begin{document}

\begin{frontmatter}
\title{A note on $\alpha$-permanent and loop soup}
\runtitle{A note on $\alpha$-permanent and loop soup}

\begin{aug}
\author[A]{\fnms{Xiaodan}~\snm{Li}\ead[label=e1]{lixiaodan@mail.shufe.edu.cn}},
\author[B]{\fnms{Yushu}~\snm{Zheng}\ead[label=e2]{yszheng666@gmail.com}}%
\address[A]{Department of Mathematics,
Shanghai University of Finance and Economics, \printead[presep={,\ }]{e1}}

\address[B]{Shanghai Center for Mathematical Sciences,
Fudan University \printead[presep={,\ }]{e2}}
\end{aug}

\begin{abstract}
In this paper, it is shown that  $\alpha$-permanent in algebra is closely related to loop soup in probability. We give explicit expansions of $\alpha$-permanents of the block matrices obtained from matrices associated to $*$-forests, which are a special class of matrices containing tridiagonal matrices. It is proved in two ways, one is the direct combinatorial proof, and the other is the probabilistic proof via loop soup.
\end{abstract}

\begin{keyword}[class=MSC]
\kwd[Primary ]{15A15}
\kwd{60G55, 60G60}
\end{keyword}

\begin{keyword}
\kwd{$\alpha$-permanent}
\kwd{Loop soup}
\end{keyword}

\end{frontmatter}


\section{Introduction}
The $\alpha$-permanent ($\alpha\in\r$) of a $d\times d$ matrix $A=(A_{ij})$ is defined as 
\[
\text{per}_\alpha(A)=\sum_{\pi\in S_d}\alpha^{\#(\pi)}\prod_{i=1}^dA_{i,\pi(i)},\]
where $S_d$ is the set of all permutations of $\{1,\cdots,d\}$ and $\#(\pi)$ denotes the number of disjoint cycles in $\pi$. The interest of such matrix functions, introduced by Vere-Jones \cite{vere,vere97}, derives from their occurrence as the coefficients in the multivariable Taylor series expansion of the determinantal form $\text{det}(I-ZA)^{-\alpha}$, where $Z=\text{diag}(z_1,z_2,\cdots,z_d)$. Specifically, 
\begin{equation}\label{detper}
\text{det}(I-ZA)^{-\alpha}=\sum_{\mathbf{q}\in\N^d}\text{per}_\alpha(A[\mathbf{q}])\prod_{i=1}^d\dfrac{z_i^{q_i}}{q_i!}.
\end{equation}
Here $\N=\{0,1,\cdots\}$, $\mathbf{q}=(q_1,\cdots,q_d)$, and $A[\mathbf{q}]$ denotes the square block matrix of order $\sum_{i=1}^d q_i$, obtained from $A$ by repeating the index $i$ $q_i$ times. The equality \eqref{detper} is in fact a reformulation of the $\alpha$-extension of MacMahon's master theorem. 

The basic problem in the study of $\alpha$-permanent is to provide an explicit expansion of $\text{per}_\alpha(A[\bq])$. Precisely, the expansion of $\text{per}_\alpha(A[\bq])$ is a linear combination of the monomial in the form $\prod_{i,j=1}^d A_{ij}^{n_{ij}}$ for some $(n_{ij})\in \N^{d\times d}$. It is of interest to figure out  the coefficients of  monomials in the expansion.
Rubak, M\o ller, and McCullagh \cite{rubak} provided explicit  expansions for block matrices of some simple types.
In this paper, we use two different ways to give explicit expansions of the $\alpha$-permanents of the block matrices obtained from matrices associated to $*$-forests (defined in Section \ref{combinsection}), which are a special class of matrices containing tridiagonal matrices. 
\begin{theorem}\label{alphaper}
Let $A[\bq]$ be the block matrix obtained from a $d\times d$ matrix $A=(A_{ij})$ associated to a $*$-forest with block sizes $\bq=(q_1,\cdots,q_d)$. Then
\begin{equation}\label{conbi}
    \text{per}_\alpha(A[\bq])=\sum_{\bn\in T_\bq(A)}\dfrac{\prod_{i=1}^d\Gamma(q_i+\alpha)q_i!}{\Gamma(\alpha)\prod_{i=1}^dn_{ii}!\cdot\prod_{1\le i<j\le d}\Gamma(n_{ij}+\alpha)n_{ij}!}\prod_{i,j=1}^d A^{n_{ij}}_{ij},
\end{equation}
where $0^0:=1$ and $T_\bq(A)\subset \N^{d\times d}$ is defined in Section \ref{combinsection}, which represents the set of non-trivial powers of the monomials in $\text{per}_\alpha(A[\bq])$.
\end{theorem}

Another focus of the paper is the loop soup, which provides an interpretation of $\alpha$-permanent in the probabilistic language. Roughly speaking, the loop soup is a Poisson ensemble of Markov loops. It has been studied intensively in the recent twenty years due to its vigorous interaction with Gaussian free field, conformal loop ensemble, uniform spanning trees, perturbed Brownian motion, etc. 
See \cite{Werner,LeJan11,lupu} for more information on loop soup.

It was indicated by Le Jan \cite{le} that the loop soup with intensity $\alpha(>0)$ is closely related to the $\alpha$-permanent.
Let $V$ be a finite set and $P=(P_{xy})_{x,y\in V}$ be a sub-Markovian transition matrix. Assuming that the corresponding discrete-time Markov chain is  transient, denote by $\cL_\alpha$ its associated unrooted oriented loop soup with intensity $\alpha$ (Cf. Section \ref{loopsection} for details). Let $\theta=\big(\theta_x\big)_{x\in V}$ be the occupation time field on vertices of $\cL_\alpha$. Namely, $\theta_x$ is the sum of the number of visits at $x$ of each loop in $\cL_\alpha$. 
The law of $\theta$ is given by the  permanental random field defined as follows. For more general definitions of permanental random fields, we refer the reader to \cite{rubak}.

\begin{theorem}\label{occupationlaw}
$\theta$ is a permanental random field with parameter $(\alpha^{-1}, P)$. Namely, for any $\bq\in\N^V$,
\begin{align*}
\p(\theta=\bq)=\det(I-P)^\alpha\cdot\dfrac{\text{per}_\alpha(P[\bq])}{\prod_{x\in V}q_x!}.
\end{align*}
\end{theorem}

\begin{remark*}
Let $N=(N_{xy})_{x,y\in V}$ be the occupation time field on edges of $\cL_\alpha$. That is, $N_{xy}$ is the sum of the number of crossings from $x$ to $y$ of each loop in $\cL_\alpha$. It holds that $\theta_x=\sum_{y\in V}N_{xy}=\sum_{x\in V}N_{yx}$ for any $x\in V$. The law of $N$ is given by the following formula (\cite[Proposition 4.1]{le}).
    For any $\mathbf{n}=(n_{xy})_{x,y\in V}$ satisfying  $\sum_{y\in V}n_{xy}=\sum_{y\in V}n_{yx}(=:q_x)$ for all $x\in V$,
    \begin{equation}\label{edgelaw}
    \p(N=\mathbf{n})=\det(I-P)^\alpha R(\mathbf{n})\cdot\dfrac{\prod_{x,y\in V}P_{xy}^{n_{xy}}}{\prod_{x\in V}q_x!},
    \end{equation}
    where $R(\mathbf{n})$ is the coefficient of $\prod_{x,y\in V}P_{xy}^{n_{xy}} $ in $\text{per}_\alpha(P[\bq])$ with $\mathbf{q}=(q_x)_{x\in V}$.
\end{remark*}

 Theorem \ref{occupationlaw} follows immediately from \eqref{edgelaw}. In this paper, we provide another proof for Theorem \ref{occupationlaw}, which gives the intuition of the appearance of the $\alpha$-permanent in the occupation law.
In the case where $P$ is associated to a $*$-forest, the law of $N$ and $\theta$ have  simple expressions as follows. 
 \begin{theorem}\label{computeoccupation}
 Suppose $P$ is associated to a $*$-forest. Let $E=E(P):=\big\{\{x,y\}:x,y\in V,\, P_{xy}\neq 0\text{ or }P_{yx}\neq 0\big\}$ $(\{x,x\}\in E$ if $P_{xx}\neq 0)$.  Then for any $\bq\in\N^V$ and $\mathbf{n}\in T_\bq(P)$, 
 \begin{align}\label{expression2}
 \p(N=\mathbf{n})=\dfrac{\det(I-P)^\alpha}{\Gamma(\alpha)}\cdot\dfrac{\prod_{x\in V}\Gamma(q_x+\alpha)}{\prod_{\{x,y\}\in E}n_{xy}!\prod_{\{x,y\}\in E:x\neq y}\Gamma(n_{xy}+\alpha)}\prod_{x,y\in V}P_{xy}^{n_{xy}},
 \end{align}
 and consequently,
 \begin{align}\label{expression22}
    &\p(\theta=\bq)=\dfrac{\det(I-P)^\alpha}{\Gamma(\alpha)}\cdot\sum_{\mathbf{n}\in T_\bq(P)}\dfrac{\prod_{x\in V}\Gamma(q_x+\alpha)}{\prod_{\{x,y\}\in E}n_{xy}!\prod_{\{x,y\}\in E:x\neq y}\Gamma(n_{xy}+\alpha)}\prod_{x,y\in V}P_{xy}^{n_{xy}}.
\end{align} 
In particular, if $P$ is associated to a forest (defined in Section \ref{combinsection}). Then for any $\bq\in \N^V$, there is exactly one element, say $\mathbf{n}_\bq$, in $T_\bq(P)$. Hence it holds that 
\[\p(\theta=\bq)=\p(N=\mathbf{n}_\bq).\]
\end{theorem}
\begin{remark*}
    The proof of Theorem \ref{computeoccupation} relies on the special structure of $*$-forest. In fact, \eqref{expression2} and \eqref{expression22} can not be generalized to general graphs except when $\alpha=1$ (one can for example check the case of the triangle graph). In \cite{le17}, Le Jan 
    showed 
    \eqref{expression2} is still true for general graphs when $\alpha=1$. This is due to the particularity of the oriented loop soups with intensity $1$. 
    In the same paper, an explicit expression for the law of unoriented crossings  was also given in the case $\alpha=1/2$, which is special for unoriented loop soups.      

    Theorem \ref{computeoccupation} is also closely related to \cite[Proposition 3.2]{Li}, which gives a conditional   occupation law of the continuous-time loop soup on trees.
\end{remark*}

In fact, the probabilistic proof of Theorem \ref{alphaper} comes directly from Theorem \ref{occupationlaw} and Theorem \ref{computeoccupation} by observing that the expansion of $\text{per}_\alpha(A)$ are polynomials in $\alpha$ and the entries of $A$.

The paper is organized as follows. In Section \ref{combinsection}, we  prove Theorem \ref{alphaper} in a combinatorial way. In Section \ref{loopsection}, we focus on the relation between the $\alpha$-permanent and the loop soup and prove Theorem \ref{occupationlaw} and  Theorem \ref{computeoccupation}.

\section{A combinatorial proof of Theorem \ref{alphaper}}\label{combinsection}
Let $A=(A_{ij})_{i,j=1}^d\in\mathbb{R}^{d\times d}$ be a square matrix. It naturally induces a graph $\bG(A)=(V(A),E(A))$ as follows: 
\begin{itemize}
\item $\bG(A)$ has $d$ vertices $V(A)=\{1,2,\cdots,d\}$;
\item for $i,j\in V(A)$, there is an edge $\{i,j\}\in E(A)$ if and only if  $A_{ij}\neq 0$ or $ A_{ji}\neq 0$;
\end{itemize}
 A graph is called a $*$-forest if it has no cycles other than self-loops. We call  a matrix $A$ associated to a $*$-forest (resp. forest) if $\bG(A)$ is a $*$-forest (resp. forest). In particular, any tridiagonal matrix is  associated to a $*$-forest.
 
To simplify notation, we omit `$(A)$' in $\bG(A)$, $V(A)$, and $E(A)$ and write $ij=ji$ for an edge $\{i,j\}\in E$.
For $\bq=(q_1,\cdots,q_d)\in\N^d$, denote by $T_\bq=T_\bq(\bG)$ the set of $\bn=(n_{ij})_{i,j=1}^d\in\N^{d\times d}$  satisfying that (1) $\text{supp}(\bn)\subset E$, i.e. $n_{ij}=0$ whenever $ij\notin E$; (2) $\bn$ is sourceless, i.e. $\sum_{j\in V}n_{ij}=\sum_{j\in V}n_{ji}$ for all $i\in V$; (3) $\sum_{j\in V} n_{ij}=q_i$ for all $i\in V$. Keep in mind that when $\mathbb{G}$ is a $*$-forest, under condition (1), condition (2) is equivalent to $n_{ij}=n_{ji}$ for all $ij\in E$.

\begin{proof}[A combinatorial proof of Theorem \ref{alphaper}]
For a $d\times d$ matrix $A$  associated to a $*$-forest and $\bq\in\N^d$, we consider every vertex $i\in V$ to have $q_i$ copies labelled by $i_1,\cdots,i_{q_i}$ and focus on the permutations of $V[\bq]:=\{i_m:i\in V\text{ and }m=1,2,\cdots,q_i\}$.
Denote by $S(V[\bq])$  the set of all permutations of $V[\bq]$. 
For $\pi\in S(V[\bq])$, the crossing of $\pi$ is an element in $\N^{V\times V}$ defined as:
\[\big(N(\pi)\big)_{ij}=\#\big\{m\in [1,q_i]:\pi(i_m)\in\{j_1,\cdots,j_{q_j}\}\big\},\text{ for }i,j\in V.\]
With the above notation, we can write
\begin{align*}
    \text{per}_\alpha(A[\bq])&=\sum_{\bn\in \N^{V\times V}}\sum_{\pi\in S(V[\bq]): N(\pi)=\bn}\alpha^{\#(\pi)}\prod_{i,j\in V}A^{n_{ij}}_{ij}\\
    &=\sum_{\bn\in T_\bq}\sum_{\pi\in S(V[\bq]):N(\pi)=\bn}\alpha^{\#(\pi)}\prod_{i,j\in V}A^{n_{ij}}_{ij},
\end{align*}
where the second equality follows from the simple fact that the second sum on the right-hand side is non-zero only when $\bn\in T_\bq$. 

Now it suffices to show that for any $d\times d$  matrix $A$ associated to a $*$-forest, $\bq\in\N^d$, and $\bn\in T_\bq$,
\begin{align}\label{induction}
    \sum_{\pi\in S(V[\bq]): N(\pi)=\bn}\alpha^{\#(\pi)}=\dfrac{\prod_{i=1}^d\Gamma(q_i+\alpha)q_i!}{\Gamma(\alpha)\prod_{i=1}^dn_{ii}!\cdot\prod_{1\le i<j\le d}\Gamma(n_{ij}+\alpha)n_{ij}!}.
\end{align}
Note that the above statement depends on $A$ only via $\bG$ ($T_\bq$ depends on $\bG$). First, we will prove it under the extra condition that $\bG$ is a forest. The proof goes by induction on $d$. In the case where $\bG$ contains no edges, the only possible choice for $\bq$ and $\bn$ is $q_i=n_{ij}=0$ for any $1\le i, j\le d$, and \eqref{induction} holds. In particular, this covers the case of $d=1$.
Assuming that \eqref{induction} holds for any forest $\bG$ with vertices $\{1,2,\cdots,d-1\}$, $\mathbf{q}\in \N^{d-1}$, and $\bn\in T_\bq$, we will prove it for $\bG$ with vertices $\{1,2,\cdots,d\}$, $\mathbf{q}\in \N^d$, and $\bn\in T_\bq$.

We exclude the trivial case where $\bG$ contains no edges. Then there exists at least one vertex with exactly one neighbour in $\bG$. Fix such a vertex $y$ and denote by $x$ its unique neighbour. It holds that $q_y=n_{xy}\le q_x$. Let $\widetilde{\bG}=(\widetilde{V},\widetilde{E})$ be the graph obtained by removing $y$ and $xy$ from $\bG$. Define $\widetilde{\bq}=(\widetilde{q}_i)_{i\in\widetilde{V}}$ and $\widetilde{\bn}=(\widetilde{n}_{ij})_{i,j\in \widetilde{V}}$ as follows:
\begin{align*}
\widetilde{q}_i=q_i-1_{\{i=x\}}q_y,\ 
    \widetilde{n}_{ij}=n_{ij},\text{ for any }i,j\in\widetilde{V}.
\end{align*}
Then $\widetilde{\bG}$ is also a forest and $\widetilde{\bn}\in T_{\widetilde{\bq}}(\widetilde{\bG})$. We shall introduce some notation. Set $Y :=\{y_1,\cdots,y_{q_y}\}$, $X :=\{x_1,\cdots,x_{q_x}\}$, and $\widetilde{V}[\widetilde{\bq}] :=\{i_m:i\in \widetilde{V}\text{ and }m=1,\cdots,\widetilde{q}_i\}$. Below, we use $S$ and $\widetilde{S}$ to denote $S(V[\bq])$ and $S(\widetilde{V}[\widetilde{\bq}])$ respectively. Let $S_Y :=\big\{\sigma=(\sigma^-,\sigma^+):\sigma^{\pm}:Y\rightarrow X \text{ are both injective}\big\}$ and for any $\sigma=(\sigma^-,\sigma^+)\in S_Y$, $S^\sigma:=\{\pi\in S:\pi^{-1}|_{Y}=\sigma^-\text{ and }\pi|_{Y}=\sigma^+\}$. 

In the following, we will show that for any $\sigma\in S_Y$, there is a one-to-one correspondence between $S^\sigma$ and $\widetilde{S}$. To this end, we relabel $X $ such that $\sigma^+(Y)=\{x_m:m\in[q_x-q_y+1,q_x]\}$. For $\pi\in S^\sigma$, define a map $\widetilde{\pi}\in\widetilde{S}$  via:  
\[\widetilde{\pi}(i_m):=\pi^{k(i_m)}(i_m), ~\forall i_m\in\widetilde{V}[\widetilde{\bq}],\]
where $\pi^k$ is the $k$-th fold composition of $\pi$ with itself and $k(i_m) :=\inf\big\{k:\pi^k(i_m)\in \widetilde{V}[\widetilde{\bq}]\big\}$. 
The idea is that the orbits induced by $\sigma$ consist of cycles and bridges\footnote{The orbits of $\sigma$ consist of the self-avoiding paths of the following two kinds: (1) the cycles $(x_{m_1},y_{m^\prime_1},\cdots,x_{m_k},y_{m^\prime_k})$ with $\sigma^-(y_{m^\prime_i})=x_{m_i}$, $\sigma^+(y_{m^\prime_{i-1}})=x_{m_i}$, and $\sigma^+(y_{m^\prime_k})=x_{m_1}$; (2) the bridges $(x_{m_1},y_{m^\prime_1},\cdots,y_{m^\prime_{k-1}},x_{m_k})$ with $x_{m_1}\notin \sigma^+(Y)$, $x_{m_k}\notin \sigma^-(Y)$, $\sigma^-(y_{m^\prime_i})=x_{m_i}$, and $\sigma^+(y_{m^\prime_{i-1}})=x_{m_i}$.},which are also sections of the orbits of $\pi$, for any $\pi\in S^\sigma$. $\widetilde{\pi}$ is obtained from $\pi$ by removing all these cycles and bridges and identifying the two endpoints of each bridge. See Figure \ref{fig} for an illustration. 
  We can readily check that $\pi\mapsto \widetilde{\pi}$ is a bijection from $S^\sigma$ to $\widetilde{S}$. Furthermore, it holds that
\[\#(\pi)=\#(\widetilde{\pi})+\#(\sigma),\]
where $\#(\sigma)$ is the number of cycles in the orbits of $\sigma$. 

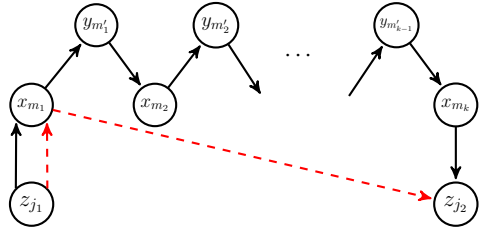
\begin{figure}
\centering
\begin{tikzpicture}[->,>=stealth',shorten >=1pt,auto,node distance=3cm,thick,main/.style={circle,draw,font=\Large\bfseries}]

  \node[main,scale=0.5] (a) {$y_{m_1^\prime}$};
  \node[main,scale=0.5] (b) [below left of=a,xshift=0.4cm] {$x_{m_1}$};
  \node[main,scale=0.53] (e) [right of =a]  {$y_{m_2^\prime}$};
   \node[main,scale=0.5] (c) [below left of =e,xshift=0.5cm] {$x_{m_2}$};
    \node[main,scale=0.5] (i) [right of =c,xshift=5cm] {$x_{m_k}$};
   \node[main,scale=0.4] (f) [right of =e,xshift=3cm]  {$y_{m_{k-1}^\prime}$};
   \node[main,scale=0.6] (d) [below of=b,yshift=0.8cm] {$z_{j_1}$};  
   \node (g)[right of =c,xshift=-1.5cm] {};
    \node (h)[right of =c,xshift=-0.5cm] {};
     \node[main,scale=0.6] (j) [below of=i,yshift=0.8cm] {$z_{j_2}$};  
   \path
    (b) edge  (a)
    (c) edge (e)
    (a) edge (c)
      (d.north east) edge[dashed,red]  (b.south east)
      (d.north west) edge (b.south west)
     (e) edge node[xshift=0.4cm]{$\cdots$}(g)
     (h) edge (f)
     (f) edge (i)
     (i) edge (j)
     (b) edge[dashed,red] (j);   
        \end{tikzpicture}
         \caption{Suppose $z$ is the parent of $x$ and the black arrows represent a section of the orbit of some $\pi\in S^\sigma$. Then the red dashed arrows represent the corresponding section of the orbit of $\widetilde{\pi}$.}\label{fig}
         \end{figure}

Following this, we have 
\begin{align}\label{calcu}
    \sum_{\pi\in S:N(\pi)=\bn}\alpha^{\#(\pi)}=\sum_{\sigma\in S_Y}\sum_{\pi\in S^\sigma}\alpha^{\#(\pi)}=\left(\sum_{\sigma\in S_Y}\alpha^{\#(\sigma)}\right)\cdot\left(\sum_{\widetilde{\pi}\in \widetilde{S}}\alpha^{\#(\widetilde{\pi})}\right).
\end{align}
By the induction hypothesis, it holds that
\begin{align}\label{term1}
    \sum_{\widetilde{\pi}\in \widetilde{S}}\alpha^{\#(\widetilde{\pi})}=\Gamma(\alpha)^{-1}\dfrac{\prod_{i\in \widetilde{V}}\Gamma(\widetilde{q_i}+\alpha)\widetilde{q}_i!}{\prod_{ij\in\widetilde{E}}\Gamma(\widetilde{n}_{ij}+\alpha)\widetilde{n}_{ij}!},
\end{align}
On the other hand, we claim that 
\begin{align}\label{term2}
    \sum_{\sigma\in S_Y}\alpha^{\#(\sigma)}=\dfrac{\Gamma(q_x+\alpha)q_x!}{\Gamma(\widetilde{q}_x+\alpha)\widetilde{q}_x!}.
\end{align}
To this end,  set $Y_m=\{y_1,\cdots,y_m\}$  and $S^m_Y=\big\{(\sigma^-,\sigma^+|_{Y_m}):(\sigma^-,\sigma^+)\in S_Y\big\}$ ($0\le m\le q_y$). We can similarly define the  orbits of $\sigma^m\in S^m_Y$. Denote by $\#(\sigma^m)$ the number of cycles in the orbit of $\sigma^m$. Observe that given $\sigma^{m-1}\in S^{m-1}_Y$,  $y_m$ can be traced, along the orbit of $\sigma^{m-1}$, back to a unique element $x_{m^*}\in X\setminus \sigma^{m-1}(Y_{m-1})$.
So if we consider the collection $\big\{\sigma^m\in S^m_Y:\sigma^m\text{ is an extension of }\sigma^{m-1}\big\}$ as further choosing the image of $y_m$, then there 
are exactly $(q_x-m+1)$ choices (since $\sigma^m(y_m)$ can be any element in $X\setminus \sigma^{m-1}(Y_{m-1})$) and among them, the choice of $\sigma^m(y_m)=x_{m^*}$ makes the number of cycles increase by $1$; while for other choices, the number remains the same. Thus,  for any given $\sigma^{m-1}\in S^{m-1}_Y$,
\[\sum_{\sigma^m}\alpha^{\#(\sigma^m)}=(q_x-m+\alpha)\cdot\alpha^{\#(\sigma^{m-1})},\]
where the sum is over $\sigma^m\in S^m_Y$ that is an extension of $\sigma^{m-1}$. Summing over $\sigma^{m-1}\in S^{m-1}_Y$, it follows that
\[\sum_{\sigma^m\in S^m_Y}\alpha^{\#(\sigma^m)}=(q_x-m+\alpha)\sum_{\sigma^{m-1}\in S^{m-1}_Y}\alpha^{\#(\sigma^{m-1})},\]
which immediately leads to \eqref{term2}.
Substituting \eqref{term1} and \eqref{term2} to \eqref{calcu}, we reach \eqref{induction}.

For a general $*$-forest $\bG$, we consider a new graph $\bG^*$ obtained from $\bG$ by changing every self-loop at $x\in V$ to an edge from $x$ to a newly-created copy of $x$. Then $\bG^*$ is a forest. It is simple to deduce \eqref{induction} for $\bG$ from that for $\bG^*$.
\end{proof}

\section{Link with loop soup}\label{loopsection}

\subsection{Proof of  Theorem \ref{occupationlaw}}
Recall that  $V$ is a finite set and $P=(P_{xy})_{x,y\in V}$ is a sub-Markovian transition matrix. Denote $P_{x\Delta}=1-\sum_{y\in V}P_{xy}$ for any $x\in V$. 
Consider the discrete-time Markov chain (DTMC) $X=\big((X_n)_{0\le n< \zeta},(\p^x)_{x\in V}\big)$ on $G$ which being at $x$, jumps to $y$ with probability $P_{xy}$ and is killed with probability $P_{x\Delta}$.
It is assumed that $X$ is transient. So its Green's function $G(x,y):=E^x\big(\sum_{n=0}^{\zeta-1}1_{\{X_n=y\}}\big)$ is finite for all $x,y\in V$. Denote $G=\big(G(x,y)\big)_{x,y\in V}$. Then it holds that $G(I-P)=(I-P)G=I$.  

In the following, $\bG$ always refers to $\bG(P)$. We give a brief introduction to the (discrete-time) oriented loop soup associated to $X$. Any discrete-time path on $\bG$ with the same starting and terminal points is a rooted loop on $\bG$. Forgetting the starting point of a rooted loop, it defines an unrooted loop on $\bG$. For an unrooted loop $\gamma$, its multiplicity $J(\gamma)$ is the  maximal integer $J$ such that $\gamma$ can be written as the concatenation of $J$ identical unrooted loops. The loop measure associated to $X$ is defined to be the measure $\mu$ on the space of unrooted loops with $\mu(\gamma)=\mu(\{\gamma\})$ equaling the product of the transition probabilities of the edges crossed by $\gamma$ divided by $J(\gamma)$ for each loop $\gamma$. The oriented loop soup associated to $X$ with intensity $\alpha(>0)$ is by definition a Poisson point process on the spaces of unrooted loops on $\bG$ with intensity $\alpha\mu$. 
We use $\cL_\alpha$ to denote the loop soup associated to $X$ with intensity $\alpha$. It is well-known that the total mass of $\mu$ is $-\log(\det(I-P))$. Thus, for a loop configuration $\mathscr{L}$ with $k$ different loops $\gamma_1,\cdots,\gamma_k$ and each loop $\gamma_i$ repeating $r_i$ times, it holds that
\begin{align}\label{soupprob}
   \p(\cL_\alpha=\mathscr{L})=\det(I-P)^\alpha\alpha^{\sum_{i=1}^k r_i}\prod_{i=1}^k\dfrac{\mu(\gamma_i)^{r_i}}{r_i!}.
\end{align}


Before the proof of  Theorem \ref{occupationlaw}, we introduce the (vertex-)extended graph of $G$.

\subsection*{Extended graph} Let $K$ be a large integer and $\mathbf{K}=(K_x)_{x\in V}$ with $K_x=K$ for all $x\in V$. 
The extended graph $\bG^K=(V^K,E^K)$ is defined as follows. We set $V^K=V[\mathbf{K}]$ and claim $x_i$ and $y_j$ are neighboured in $\bG^K$ if $x$ and $y$ are neighboured in $\bG$. Let $X^K$ be the DTMC on $\bG^K$ with transition matrix $P^K_{x_i,y_j}:=P_{xy}/K$ and $\cL_\alpha^K$ be the oriented loop soup with intensity $\alpha$ associated to $X^K$. Since $X$ has the law of the projection of $X^K$ on $\bG$, the projection of $\cL_\alpha^K$ on $\bG$ is distributed as $\cL_\alpha$.

\begin{proof}[Proof of  Theorem \ref{occupationlaw}] Note that the probability that $\cL^K$ visits any vertex at most once goes to $1$ as $K\rightarrow\infty$.  Therefore, we can focus on the event that $\cL^K\in \mathcal{C}^*$, where $\mathcal{C}^*$ is the collection of the loop configurations on $\bG^K$ that visit every vertex at most once.
By the projection relation of $\cL$ and $\cL^K$, for any $\mathbf{n}\in T_\bq$, it holds that
\begin{align}\label{computecross}
\begin{split}
    \p(\theta=\mathbf{q})&=\sum_{\mathscr{L}\in \mathcal{C}^*: \theta_\bG(\mathscr{L})=\mathbf{q}}\p(\cL_\alpha^K=\mathscr{L})+o_K(1)\\
   &=\det(I-P)^\alpha K^{-\sum\limits_{x\in V}q_x}\sum_{\bn\in T_\bq}\sum_{\mathscr{L}\in \mathcal{C}^*: N_\bG(\mathscr{L})=\bn}\alpha^{\#\mathscr{L}} \prod_{x,y\in V}P_{xy}^{n_{xy}}+o_K(1),\\ 
\end{split}
\end{align}
where 
$\#\mathscr{L}$ is the number of loops in $\mathscr{L}$,  $o_K(1)$ is a term that goes to $0$ as $K\rightarrow\infty$, and $\theta_\bG(\mathscr{L})$ (resp. $N_\bG(\mathscr{L})$) is the projection on $\mathbb{G}$ of the occupation time field on vertices (resp. edges) of $\mathscr{L}$. 
The second equality in \eqref{computecross} follows from \eqref{soupprob} and the fact that the multiplicities of the loops in $\mathscr{L}\in \mathcal{C}^*$ are all $1$.

Now it remains to deal with the term \[\sum_{\bn\in T_\bq}\sum_{\mathscr{L}\in \mathcal{C}^*: N_\bG(\mathscr{L})=\bn}\alpha^{\#\mathscr{L}} \prod_{x,y\in V}P_{xy}^{n_{xy}},\] which can be computed as follows:
\begin{enumerate}[(a)]
\item first choose $q_x$ different vertices from $\{x_1,\cdots,x_K\}$ for each $x$. The total number of ways is $\prod_{x\in V}\dfrac{K!}{q_x!(K-q_x)!}$;
\item\label{bbb} assign each loop configuration $\mathscr{L}\in \mathcal{C}^*$ a weight $\alpha^{\#\mathscr{L}}\prod_{x,y\in V}P_{xy}^{n_{xy}}$, where $\bn=N_\bG(\mathscr{L})$. Sum  the weights of all the loop configurations  $\mathscr{L}$  that satisfy $\theta_\bG(\mathscr{L})=\bq$ and visit exactly the chosen vertices.
\end{enumerate}

As $K\rightarrow\infty$, $\dfrac{K!}{q_x!(K-q_x)!}\sim \dfrac{K^{q_x}}{q_x!}$. The summation in \eqref{bbb} equals $\text{per}_\alpha(P[\bq])$.
 Therefore, as $K\rightarrow\infty$, the limit of the right-hand side of \eqref{computecross} is 
\[
\det(I-P)^\alpha\cdot\dfrac{\text{per}_\alpha(P[\bq])}{\prod_{x\in V}q_x!},
\]
which completes the proof.
\end{proof}

\subsection{Proof of Theorem \ref{computeoccupation}}
Now let us turn to the proof of Theorem \ref{computeoccupation}. Since \eqref{expression22} follows directly from \eqref{expression2}, it suffices to prove \eqref{expression2}. Without loss of generality, we further assume $\bG(P)$ is connected (otherwise we consider its connected components separately). Then $\bG$ is a $*$-tree (i.e. a connected $*$-forest).
First, we impose the extra conditions that
\begin{enumerate}[(a)]
    \item\label{extraa} $\bG$ is a tree,
    \item\label{extrab} $X$ can only be killed at some vertex $x_0$.
\end{enumerate}
Let us show \eqref{expression2} under \eqref{extraa} and \eqref{extrab}. We view $x_0$ as the root of the tree $\bG$ henceforth. Let $\mathfrak{C}_x$ be the set of children of $x$ for $x\in V$ and $\mathfrak{p}_x$ be the parent of $x$ for $x\in V\setminus \{x_0\}$.  Denote by NB$(r,p)$, Multi$(m,\mathbf{p})$, and NM$(r,\mathbf{p})$ the negative binomial distribution with parameter $(r,p)$, the multinomial distribution with parameter $(m,\mathbf{p})$, and the negative multinomial distribution with parameter $(r,\mathbf{p})$ respectively\footnote{For $r>0$, $d\ge 1$ and $\mathbf{p}=(p_1,\cdots,p_d)\in [0,1)^d$ with $\norm{\mathbf{p}}_1<1$,  NM$(r,\mathbf{p})$ is a distribution on $\N^d$ with probability mass function:
\[\Gamma(\norm{\mathbf{n}}_1+r)\frac{(1-\norm{\mathbf{p}}_1)^r}{\Gamma(r)}\prod_{i=1}^d\frac{p_i^{n_i}}{n_i!},\quad \mathbf{n}\in\N^d.\]
For $r>0$ and $0\le p<1$, NB$(r,p)$ is just NM$(r,\mathbf{p})$ with $\mathbf{p}=(p)$.}. We shall prove a preliminary lemma.

\begin{lemma}
Under the above setting, let $\mathbf{P}_x=(P_{xy}:y\in \mathfrak{C}_x)$. Then
\begin{enumerate}[(i)]
    \item\label{nm1} $\theta_{x_0}$ follows a NB$(\alpha,1-P_{x_0\Delta})$ distribution; $\big(N_{x_0x}:x\in\mathfrak{C}_{x_0}\big)$ follows a NM$(\alpha,\mathbf{P}_{x_0})$ distribution;
    \item\label{nm2} for any $x\in V\setminus \{x_0\}$, conditionally on $N(x\mathfrak{p}_x)$, $\big(N_{xy}:y\in\mathfrak{C}_x\big)$ follows a NM$\big(N_{x\mathfrak{p}_x}+\alpha,\mathbf{P}_x\big)$ distribution. 
\end{enumerate}
\end{lemma}
\begin{remark*}
    The conclusion \eqref{nm1} can be generalized to any transition matrix $P$, while \eqref{nm2} relies on the tree structure of $\bG$.
\end{remark*}
\begin{proof}
In this proof, we will frequently use the following simple fact.
\begin{fact}[NB-Multi mixture=NM]\label{fact}
If $Y$ is a NB$(r,p)$ random variable and conditionally on $Y$, $X$ follows the Multi$(Y,\mathbf{p})$ distribution, then the unconditional distribution of $X$ is NM$(r,p\mathbf{p})$, where $p\mathbf{p}=(pp_1,\cdots,pp_d)$.
\end{fact}

For \eqref{nm1}, the law of $\theta_{x_0}$ follows from the standard result of the marginal distribution of a random permanent field (Cf. \cite[\S3.1]{rubak}). It can be easily deduced that the paths between consecutive visits of $x_0$ in the loops in $\cL_\alpha$ are i.i.d and each of them is distributed as an excursion of $X$ at $x_0$. In other words, conditionally on $\theta_{x_0}$, if we cut off the loops visiting $x_0$ at every visit at $x_0$, then what we get is just $\theta_{x_0}$ independent excursions of $X$ at $x_0$. So the conditional law of $\big(N_{x_0x}:x\in\mathfrak{C}_{x_0}\big)$ is Multi$\big(\theta_{x_0},(1-P_{x_0\Delta})^{-1}\mathbf{P}_{x_0}\big)$. Their unconditional law follows from Fact \ref{fact}.

For \eqref{nm2}, for any $y\in \mathfrak{C}_x$, we divide $N_{xy}$ into two parts: $N_{xy}=N^{(1)}_{xy}+N^{(2)}_{xy}$, where $N^{(1)}_{xy}$ (resp. $N^{(2)}_{xy}$) is the number of crossings from $x$ to $y$ by the loops in $\cL_\alpha$ visiting $\mathfrak{p}_x$ (resp. not visiting $\mathfrak{p}_x$). For the first part, conditionally on $N_{x\mathfrak{p}_x}=m$, the trace of the loops visiting $\mathfrak{p}_x$ on
the branch\footnote{A branch at $x$ is defined as a connected component of the tree $G$ when removing the vertex $x$, to which we add $x$.} at $\mathfrak{p}_x$ containing $x$ consists of exactly $m$ excursions at $\mathfrak{p}_x$. As the previous arguments, these excursions are i.i.d. and each of them has the same law as an excursion of $X$ at $\mathfrak{p}_x$ conditioned to hit $x$. So the number of crossings from $x$ to its children by each excursion follows a NB$(1,1-P_{x\mathfrak{p}_x})$ distribution (i.e. geometric distribution with success probability $P_{x\mathfrak{p}_x}$).  
The sum of them, i.e. $\sum_{y\in\mathfrak{C}_x}N^{(1)}_{xy}$, is a NB$(m,1-P_{x\mathfrak{p}_x})$ random variable. For the second part, it is easy to deduce that the loops not visiting $\mathfrak{p}_x$ form a loop soup associated to $X$ killed at $\mathfrak{p}(x)$. Hence it follows from \eqref{nm1} that $\sum_{y\in\mathfrak{C}_x}N^{(2)}_{xy}$ has the NB$(\alpha,1-P_{x\mathfrak{p}_x})$ distribution. The independence of $\big\{N^{(1)}_{xy}:y\in \mathfrak{C}_x\big\}$ and $\big\{N^{(2)}_{xy}:y\in \mathfrak{C}_x\big\}$ yields that the conditional distribution of $\sum_{y\in \mathfrak{C}_x}N_{xy}$ is NB$(m+\alpha,1-P_{x\mathfrak{p}_x})$. Moreover, we can readily see from the above arguments that conditionally on $\sum_{y\in \mathfrak{C}_x}N_{xy}$,
$\big(N_{xy}:y\in\mathfrak{C}_x\big)$ follows a Multi$\big(\sum_{y\in\mathfrak{C}_x}N_{xy},(1-P_{x\mathfrak{p}_x})^{-1}\mathbf{P}_x\big)$ distribution. Thus, Fact \ref{fact} implies \eqref{nm2}.
\end{proof}

Iteratively using \eqref{nm1} and \eqref{nm2}, we get
\begin{align*}
    \p(N=\bn)&=\Gamma(q_{x_0}+\alpha)\frac{P_{x_0\Delta}^\alpha}{\Gamma(\alpha)}\prod_{x\in\mathfrak{C}_{x_0}}\frac{P_{x_0x}^{n_{x_0x}}}{n_{x_0x}!}\prod_{x\in V\setminus\{x_0\}}\left[\Gamma(q_x+\alpha)\frac{P_{x\mathfrak{p}_x}^{n_{x\mathfrak{p}_x}+\alpha}}{\Gamma(n_{x\mathfrak{p}_x}+\alpha)}\prod_{y\in\mathfrak{C}_x}\frac{P_{xy}^{n_{xy}}}{n_{xy}!}\right]\\
    &=\dfrac{\det(I-P)^\alpha}{\Gamma(\alpha)}\cdot\dfrac{\prod_{x\in V}\Gamma(q_x+\alpha)}{\prod_{xy\in E}\big[n_{xy}!\Gamma(n_{xy}+\alpha)\big]}\prod_{x,y\in V}P_{xy}^{n_{xy}},
\end{align*}
where in the second equality, we use
\begin{align}\label{Gformula}
\det(I-P)=\det(G)^{-1}=\left[\prod_{j=0}^{d-1}G_{V\setminus\{x_0,\cdots,x_{j-1}\}}(x_j,x_j)\right]^{-1},
\end{align}
where $\{x_0,\cdots,x_{d-1}\}$ is an enumeration of $V$ such that $\mathfrak{p}_{x_j}\in\{x_0,\cdots,x_{j-1}\}$ for any $j=1,\cdots,d-1$ and $G_{V\setminus\{x_0,\cdots,x_{j-1}\}}$ is the Green's function of  $X$ killed at $\{x_0,\cdots,x_{j-1}\}$. The second equality is a well-known formula of the determinant of the Green's function (see for example \cite[Proposition 1.31]{Werner}). 

Now we proceed to the general cases. First, we remove condition \eqref{extrab}. Still fix some vertex $x_0\in V$. We consider the $h$-transform $X^h$ of $X$ using the excessive function
\[h(x)=\p^x(T_{x_0}<\infty),~~x\in V,\]
where $T_{x_0}=\inf\{n\ge 0:X_n=x_0\}$. Since the law of the loop soup is invariant under the transform (because the loop measure is invariant under the transform) and $X^h$ is an irreducible transient DTMC that can only be killed at $x_0$, it boils down to the previous case and we have \eqref{expression2} holds if $\det(I-P)$ is substituted by $\det(I-P^h)$, where $P^h_{xy}=P_{xy}\frac{h(y)}{h(x)}$ is the transition matrix of $X^h$. Further using that the diagonal entries in the Green's function (seen as a matrix) are invariant under the transform, we can readily deduce from \eqref{Gformula} that $\det(I-P)=\det(I-P^h)$. That gives rise to \eqref{expression2}. 

Next, we further remove condition \eqref{extraa}. Similar to the combinatorial proof of Theorem \ref{alphaper}, we add to the vertex set a copy of every vertex and define a DTMC $X^*$ on the extended vertex set via its transition function:
\begin{align*}
    P^*_{xy}=\begin{cases}P_{xy},&\text{ if } x,y\in V;\\
    P_{xx},&\text{ if }x\in V\text{ and } y=x^*;\\
    1,&\text{ if }y\in V\text{ and } x=y^*;\\
    0,&\text{ otherwise},
    \end{cases}
\end{align*}
where $x^*$ is the copy of $x$ in the extended vertex set. Then $\bG(P^*)$ is a tree and the projection of $X^*$ on $V$ has the same law as $X$, where the projection maps every path $(x,x^*,x)$ to the path $(x,x)$ for any $x\in V$. It follows that the loop soup associated to $X$ has the same law as the projection of the loop soup associated to $X^*$. Again, it reduces to the previous case and \eqref{expression2} holds if $\det(I-P)$ is substituted by $\det(I-P^*)$. By \eqref{Gformula}, it is easily seen that $\det(I-P)=\det(I-P^*)$, which leads to \eqref{expression2}. That completes the proof. 

\begin{funding}
The first author is supported by the Fundamental Research Funds for the Central Universities. The second  author is partially supported by NSFC, China (No. 11871162).
\end{funding}



\end{document}